\documentclass[11pt]{article}

\usepackage{fullpage}
\usepackage[utf8]{inputenc}
\usepackage[T1]{fontenc}
\usepackage{amsmath,amsfonts,amssymb,amsthm,bbm}
\usepackage{stackrel}
\usepackage{enumerate}
\usepackage{graphicx}
\usepackage[hidelinks]{hyperref}
\usepackage{subfigure}
\usepackage{color}
\usepackage{stmaryrd}

\newtheorem{theorem}{Theorem}
\newtheorem{corollary}[theorem]{Corollary}
\newtheorem{proposition}[theorem]{Proposition}

\newtheorem{lemma}[theorem]{Lemma}
\newtheorem{definition}[theorem]{Definition}

\newtheorem{remark}[theorem]{Remark}

\numberwithin{theorem}{section}
\numberwithin{equation}{section}
\numberwithin{figure}{section}

\newcommand{\Lc}{\mathcal{L}}

\newcommand{\Pb}{\mathbb{P}}

\newcommand{\Rb}{\mathbb{R}}

\newcommand{\Db}{\mathbb{D}}

\title{Generalized intersection exponents and local cut points for three-dimensional Brownian loop soup}

\author{
Yifan Gao$^{1}$\thanks{$^{1}$Institute for Theoretical Sciences, Westlake University, \href{mailto:gaoyifan75@westlake.edu.cn}{gaoyifan75@westlake.edu.cn}}
\and
Ruixuan Li$^{2}$\thanks{$^{2}$School of Mathematical Sciences, Peking University, \href{mailto:2300010605@stu.pku.edu.cn}{2300010605@stu.pku.edu.cn}}
\and
Xinyi Li$^{3}$\thanks{$^{3}$Beijing International Center for Mathematical Research, Peking University, \href{mailto:xinyili@bicmr.pku.edu.cn}{xinyili@bicmr.pku.edu.cn}}
}

\date{\today}

\begin{document}
\maketitle

\begin{abstract}
We study generalized non-intersection probabilities for the
three-dimensional Brownian loop soup at subcritical intensities.  We establish
the existence of generalized intersection exponents (GIE) and prove an
up-to-constants estimate for these probabilities by means of a separation
lemma tailored to this setting. We also relate the Hausdorff dimension of the
set of local cut points of the three-dimensional Brownian loop soup to the
GIE, and show that the GIE is continuous at intensity zero, where it reduces
to the classical Brownian intersection exponent. In particular, this
implies that, for sufficiently small intensity parameters, the set of local
cut points has Hausdorff dimension strictly larger than $1$.
\end{abstract}

\section{Introduction}
The Brownian loop soup (BLS), introduced by Lawler and Werner \cite{LW04}, is a
Poissonian ensemble of Brownian loops. In two dimensions the BLS is conformally
invariant and is closely related to Schramm--Loewner evolutions (SLE)
\cite{lawler2003conformal} and conformal loop ensembles \cite{SW12}. These
connections make the planar BLS a central object in the study of conformally
invariant random geometry. In dimension three, however, the SLE/LQG machinery
is no longer available. As a result, many geometric questions about
the three-dimensional BLS remain comparatively less understood. Recent works
have begun to clarify the percolative and geometric behavior of
the three-dimensional BLS and related loop percolation models; see, for
instance, \cite{werner2020clusters,CD2,JL26,CS16,chang2024percolation,CD1}.

A classical approach to fine geometric properties of Brownian paths is
through non-intersection probabilities. The associated intersection exponents
measure the decay rate of the probability that independent Brownian paths avoid
one another across large annuli. In dimension two, Lawler, Schramm and Werner
computed these exponents using SLE \cite{LSW2001a,LSW2001b,LSW2002a,LSW2002b},
which led to exact Hausdorff dimensions for cut points, pioneer points, and the
frontier. In dimension three, the exact values of the corresponding exponents
remain almost completely unknown.
Nevertheless, Lawler expressed the Hausdorff dimension of Brownian cut points
in three dimensions in terms of the exponent \(\xi(1,1)\), which governs the
non-intersection of two independent Brownian motions \cite{Lawler96}; see also
\cite{BL90,Lawler98,BMinvar,gao2024boundary} for estimates and structural properties of these
exponents.

The BLS setting leads to a natural analogue of this
non-intersection picture. If one zooms in near a point that is locally pivotal for a BLS
cluster, the underlying loop supplies two Brownian-type
paths entering and leaving a small neighborhood of the point. For such a local separation to persist, these two paths should not become
connected after one attaches the BLS clusters that they hit. Thus the
relevant event is no longer that one Brownian path avoids another, but
rather that a Brownian path avoids another path enlarged by an independent
BLS environment.
The main point of the present paper is to develop an analogue of the
intersection theory for this local picture.

The generalized intersection exponents introduced below quantify the decay
of these generalized non-intersection probabilities, following the terminology
of \cite{qian2021generalized,GLQ26}. They reduce to the classical Brownian
intersection exponents when the BLS intensity is zero, while for positive
intensity they encode the additional connectivity created by the BLS. Our
first main result is an up-to-constants estimate for these probabilities in
\(\mathbb R^3\). Its proof relies on a BLS version of Lawler's separation
lemma \cite{Lawler98}: unlike in the classical Brownian case, one must separate
a Brownian path from another path together with the random BLS clusters
that it intersects.

We then apply these generalized exponents to the geometry of
three-dimensional BLS clusters. The objects of interest are
\textit{local cut points}, namely points whose removal disconnects a BLS
cluster inside some neighborhood. Their dimension is obtained by a first- and
second-moment argument, in the spirit of \cite{Lawler96} for Brownian cut
points and \cite{GLQ26} for the planar BLS. Finally, we prove that
the generalized intersection exponent is continuous at intensity zero, where it
reduces to the classical Brownian intersection exponent. Combined with the
bound \(\xi(1,1)<1\) from \cite{BL90}, this implies that, for sufficiently
small intensity parameters, both the Hausdorff dimension of the set of local
cut points and the generalized percolation dimension are strictly larger than
\(1\).

\subsection{Generalized intersection exponents in three dimensions}
Let \(\Lc^\alpha\) denote the three-dimensional BLS with
intensity \(\alpha>0\); see Section \ref{subsec:BLS} for the definition. With a slight abuse of notation,
we also write \(\Lc^\alpha\) for its trace. Define the percolation probability
\begin{equation}
    \Pi(\alpha)
    :=
    \Pb(\text{there exists an unbounded cluster in }\Lc^\alpha).
\end{equation}
It was shown in \cite[Theorem~1.2]{JL26} that the three-dimensional BLS
has a non-trivial phase transition, namely
\[
    \alpha_c:=\inf\{\alpha>0:\Pi(\alpha)>0\}\in (0,\infty).
\]
Moreover, \cite[Theorem~1.2]{CD2} implies that the BLS does not
percolate at intensity \(\frac12\), i.e.\ \(\Pi(\frac12)=0\), and hence
\(\alpha_c\ge \frac12\). We write
\begin{equation}
    I_0:=\{\alpha>0:\Pi(\alpha)=0\},
\end{equation}
for the set of subcritical intensities. By monotonicity in \(\alpha\),
\(I_0\) is an interval of the form \((0,\alpha_c)\) or \((0,\alpha_c]\); it is
currently unknown whether \(\alpha_c\in I_0\). Throughout the paper we work in
the subcritical regime \(\alpha\in I_0\), that is, \(\Pi(\alpha)=0\).
 
We now describe the setup for Brownian loop soup. In a BLS \(\Lc\), two loops \(\ell,\ell'\in\Lc\)
are connected if there is a chain of loops
\(\ell_1,\ldots,\ell_n\in\Lc\) with \(\ell_1=\ell\), \(\ell_n=\ell'\), and
\(\ell_i\cap\ell_{i+1}\ne\emptyset\) for \(i=1,\ldots,n-1\). A cluster is a
maximal connected collection of loops. For sets \(A,B\subseteq\mathbb R^3\),
we write
\[
    \{A \overset{\Lc^\alpha}{\longleftrightarrow} B\}
\]
for the event that some cluster of \(\Lc^\alpha\) intersects both \(A\)
and \(B\).

Let \(\mathcal B_s(x)\) be the ball of radius \(s\) centered at \(x\), and
write \(\mathcal B_s=\mathcal B_s(0)\). Set
\(S_r=\partial\mathcal B_{e^r}\), and let \(\Lc_r\) be the collection of loops
of \(\Lc\) contained in \(\mathcal B_{e^r}\). Let
\(B,B^1,\ldots,B^k\) be \(k+1\) independent three-dimensional Brownian motions,
started from independent uniformly chosen points on \(S_0\). For a set \(V\),
write \(\Lambda(V)\) for the union of \(V\) with all BLS clusters that
intersect \(V\). Define the hitting times
\[
    T_r^i = \inf \{ t > 0 : B_t^i \in S_r\},
    \qquad
    T_r= \inf \{t > 0: B_t \in S_r\}.
\]
Let \(\overline{B_r^i}=B^i[0,T_r^i]\) and
\(\overline{B_r}=B[0,T_r]\). Let \(\mathcal F_r\) be the \(\sigma\)-field
generated by the \(k\) paths
\(\overline{B_r^1},\ldots,\overline{B_r^k}\) and by
\(\Lc_r\setminus\Lc_0\). Let \(\Lambda_r\) be the union of
\(\overline{B_r^1},\ldots,\overline{B_r^k}\) and all clusters in
\(\Lc_r\setminus\Lc_0\) that intersect these paths. Thus \(\Lambda_r\) is the
BLS enlargement of the \(k\) Brownian paths. Define
\[
    Z_r
    =
    \Pb(\overline{B_r} \cap \Lambda_r = \emptyset \mid \mathcal{F}_r),
\]
which is the conditional probability that an additional independent
Brownian path avoids this enlarged obstacle up to radius \(e^r\). For
\(\alpha \in I_0\), \(\lambda>0\), \(k \in \mathbb{N}\), and \(r\ge0\), define the \textit{generalized non-intersection probability} by
\[
    p(\alpha,k,r,\lambda)=\mathbb{E}[Z_r^\lambda].
\]

A classical submultiplicativity argument, recalled in \eqref{eq:sub},
shows that the following limit exists:
\begin{equation}\label{GIE}
    \xi_\alpha(k,\lambda)
    :=
    -\lim_{r\to\infty}\frac{1}{r}\log p(\alpha,k,r,\lambda)
    \in[0,\infty].
\end{equation}
We call \(\xi_\alpha(k,\lambda)\) the \textit{generalized intersection
exponent}.

The first main result is the following uniform up-to-constants estimate, the
BLS analogue of Lawler's estimate for Brownian intersection exponents
\cite{Lawler96}.
\begin{theorem}\label{thm:up-to-constant}
    Let $\alpha_*\in I_0$, $k \in \mathbb{N}$, and $0<\lambda_0<\lambda_1$. Then there exist
    constants $C_1,C_2>0$, depending only on $\alpha_*$, $k$, $\lambda_0$ and $\lambda_1$, such that
    for all $\alpha\in[0,\alpha_*]$, all $\lambda \in [\lambda_0,\lambda_1]$, and all $r\ge 0$,
    \begin{equation}\label{eq:up-to-constant}
        C_1 e^{-r\xi_\alpha (k,\lambda)} \leq p(\alpha,k,r,\lambda) \leq C_2 e^{-r\xi_\alpha (k,\lambda)}.
    \end{equation}
    As a result, \(\xi_\alpha(k,\lambda)\in(0,\infty)\) for all \(\alpha\in I_0\) and \(\lambda>0\).
\end{theorem}
\begin{remark}
For $\alpha=0$, Theorem~\ref{thm:up-to-constant} reduces to the
classical up-to-constants estimate for Brownian intersection exponents,
which was proved in \cite{Lawler96}. Hence it suffices to prove
Theorem~\ref{thm:up-to-constant} for $\alpha\in(0,\alpha_*]$.
\end{remark}

\paragraph{Relation to the metric graph of \(\mathbb Z^3\).}
It is useful to compare this continuum setup with recent results for the
metric graph \(\widetilde{\mathbb Z}^3\) at critical intensity \(\frac12\)
\cite{CD1,CD2}. There the isomorphism between the BLS and the
Gaussian free field gives a particularly effective description of macroscopic
clusters. Arm events provide the natural language for pivotal geometry: a
one-arm event asks for a connection to macroscopic distance, while a cut or
pivotal edge requires two separated macroscopic connections from the two sides
of the edge. Thus the dimension of cut edges is equal to $\frac12$ \cite{CD2}, which is governed by a two-arm exponent \(\frac52\) on $\widetilde{\mathbb Z}^3$ \cite{cai2025heterochromatic}.

The exponent in the present paper has a closer metric-graph analogue. We start
two independent Brownian motions on \(\widetilde{\mathbb Z}^3\) from \(0\) and
\(1\), stop them when they reach distance $N$, and attach to each path
all critical BLS clusters on $\widetilde{\mathbb{Z}^3}$ that it intersects. Werner's switching identity
for the cable-graph BLS \cite{WernerSwitching25} relates the creation of a
macroscopic Brownian arm to the (cluster) one-arm event. Consequently, at the level of exponents,
the separation cost for these two already-given Brownian arms is obtained by
subtracting the cost of the two arms from the two-arm exponent. Since the
metric-graph one-arm exponent in \(d=3\) is $\frac12$ \cite{cai2024one,drewitz2025critical}, the corresponding
generalized non-intersection exponent on \(\widetilde{\mathbb Z}^3\) is
\[
    \frac52 - 2\cdot \frac12 = \frac32.
\]
This comparison is close in spirit to the local picture above,
but it does not provide the continuum estimates needed here. For the BLS on
\(\mathbb R^3\), the corresponding one-arm and two-arm exponents are
not presently known (even at intensity $\frac12$).

\subsection{Hausdorff dimension of local cut points and percolation dimension}
We next apply the generalized exponent to a BLS analogue of
Brownian cut points. As mentioned, Lawler proved in \cite{Lawler96} that the Hausdorff
dimension of cut points of three-dimensional Brownian motion is governed by the
classical intersection exponent. We first give the definition of local cut points on a loop in the BLS.
\begin{definition}[Local cut points]
For a loop \(\gamma\in\Lc^\alpha\), a point \(x\in\gamma\), and
\(\epsilon>0\), let \(\mathcal U_\epsilon^\gamma(x)\) be the connected component of
\(\gamma\cap\mathcal B_\epsilon(x)\) containing $x$. Let
\(\mathcal C_\epsilon^\gamma(x)\) be the union of
\(\mathcal U_\epsilon^\gamma(x)\) together with all clusters of $\Lc_{\mathcal{B}_{\epsilon}(x)}^{\alpha}$ it intersects. We say that \(x\) is a
local cut point of \(\Lc^\alpha\) on \(\gamma\) if there exists
\(\epsilon>0\) such that
\(\mathcal C_\epsilon^\gamma(x)\setminus\{x\}\) is no longer connected.
\end{definition}
For \(\gamma\in\Lc^\alpha\), let \(G_{\rm loc}^\gamma\) denote the set of
local cut points on \(\gamma\), and set
\[
    G_{\rm loc}:=\bigcup_{\gamma\in\Lc^\alpha}G_{\rm loc}^\gamma .
\]
We call \(G_{\rm loc}\) the set of local cut points of the BLS.
\begin{theorem}\label{dimension}
    Let \(\alpha \in I_0\). Almost surely,
    \[
        {\rm dim}_{\mathcal{H}}(G_{\rm loc})
        =
        \max\{2-\xi_{\alpha}(1,1),0\},
    \]
    where \(\xi_{\alpha}(1,1)\) is defined in \eqref{GIE}.
\end{theorem}

We next prove that the generalized intersection exponent is continuous at
\(\alpha=0\). In particular, this implies that local cut points exist for
\(\Lc^\alpha\) when \(\alpha>0\) is sufficiently small.

\begin{theorem}\label{theorem: continuity of intersection exponent}
    The generalized intersection exponent is continuous at \(\alpha=0\): for
    every \(k\in\mathbb N\) and \(\lambda>0\),
    \[
        \lim_{\alpha\downarrow 0} \xi_{\alpha}(k,\lambda)=\xi(k,\lambda),
    \]
    where \(\xi(k,\lambda)\) is the intersection exponent of
    three-dimensional Brownian motion introduced in \cite{Lawler96}.
    In particular, there exists \(\alpha_1>0\) such that for all
    \(\alpha \in (0,\alpha_1]\),
    \[
        {\rm dim}_{\mathcal{H}}(G_{\rm loc})>1
        \qquad\text{a.s.}
    \]
\end{theorem}
We also record an application to paths supported by a Brownian trajectory
and the BLS clusters it touches. Let \(\widetilde{B}\) be a standard
three-dimensional Brownian motion with \(\widetilde{B}(0)=0\), and let
\({\cal L}^{\alpha}\) be an independent BLS with intensity \(\alpha>0\). Recall
that \(\Lambda(\widetilde{B}([0,1]))\) is the union of \(\widetilde{B}([0,1])\) and all
BLS clusters it intersects. Let \(\Gamma\) be the set of continuous curves
\(\gamma:[0,1] \to \Rb^3\) such that \(\gamma(0)=0\),
\(\gamma(1)=\widetilde{B}(1)\), and
\[
    \gamma([0,1]) \subseteq \Lambda(\widetilde{B}([0,1])) .
\]
We define the generalized percolation dimension of \(\widetilde{B}\) by
\begin{equation}
    \zeta=\inf_{\gamma\in\Gamma}\dim_{\mathcal{H}}(\gamma),
\end{equation}
where \(\dim_{\mathcal{H}}(\gamma)\) denotes the Hausdorff dimension of
\(\gamma([0,1])\).
\begin{corollary}\label{coro:percolation dimension}
    There exists \(\alpha_1>0\) such that for all
    \(\alpha\in(0,\alpha_1)\), the generalized percolation dimension satisfies
    \[
        \zeta>1 \qquad \text{a.s.}
    \]
\end{corollary}

We now give some comments on proof strategy and organization of this paper.
\begin{itemize}
    \item \textbf{Section 2:} We review basic properties of the BLS and Brownian motion, and prove several classical lemmas used later.

    \item \textbf{Section 3:} We prove Theorem~\ref{thm:up-to-constant}. The lower bound follows from the strong Markov property, whereas the upper bound is more delicate and relies on a separation lemma. The separation-lemma framework developed in \cite{GPS13} is not directly applicable here, because it relies strongly on the ability to perform surgery on both random objects. Inspired by \cite{Lawler98}, we establish a separation lemma tailored to the three-dimensional BLS.

    \item \textbf{Section 4:} We prove Theorem~\ref{dimension} using standard first- and second-moment arguments. We first derive the Hausdorff dimension of local cut points on a single Brownian loop in the presence of an independent BLS, and then extend the result to the entire BLS.

    \item \textbf{Section 5:} We prove
Theorem~\ref{theorem: continuity of intersection exponent}. The only additional
input is a small-intensity cluster estimate, obtained from the comparison
between the three-dimensional BLS and Mandelbrot fractal percolation.
In particular, we show that the Hausdorff dimension of local cut points is strictly larger than $1$ whenever $\alpha \in (0,\alpha_1)$ for some $\alpha_1>0$.
\end{itemize}

\noindent {\bf Acknowledgements:} 
We thank Runsheng Liu for helpful discussions. XL is supported by 
National Key R\&D Program of China (No.~2021YFA1002700).

\section{Preliminaries}
\subsection{Notation}
Throughout the paper, we work in three-dimensional Euclidean space
\(\Rb^3\). For
\(x\in\Rb^3\), write \(|x|\) for its Euclidean norm. For two non-empty
subsets \(A,B\subseteq \Rb^3\), define
\[
    \operatorname{dist}(A,B)
    :=
    \inf\{|x-y|:x\in A,\ y\in B\}.
\]
We also write \(\operatorname{diam}(A)\) for the Euclidean diameter of
\(A\).

For \(x\in\Rb^3\) and \(r>0\), let \(\mathcal B_r(x)\) denote the ball of
radius \(r\) centered at \(x\), and write \(\mathcal B_r=\mathcal B_r(0)\)
when the center is the origin. Let \(\Db(x)=\mathcal B_1(x)\) be the unit ball centered at $x$ and $\Db=\Db(0)$. We use logarithmic notation for spheres:
\[
    S_r:=\partial \mathcal B_{e^r}.
\]
For \(r<s\), let
\[
    \mathcal A(r,s)
    :=
    \mathcal B_{e^s}\setminus \overline{\mathcal{B}}_{e^r}
\]
be the open annulus between \(S_r\) and \(S_s\). When no confusion can arise, a
path or a loop is identified with its trace. Let $B(t)$ denote standard
three-dimensional Brownian motion. For $x\in \Rb^3$, write $\Pb^x$ for the law
of standard three-dimensional Brownian motion starting from $x$. For a set
$A\subseteq \Rb^3$, let $\tau(A)$ denote the hitting time of $A$, and let
$\sigma(A)$ denote the exit time from $A$.

\subsection{The Brownian loop soup}\label{subsec:BLS}
In this subsection, we review and prove some basic facts about the
Brownian loop soup (BLS). The BLS was introduced by Lawler and Werner in
\cite{LW04}. In $d$-dimensional Euclidean space $\mathbb{R}^d$, the Brownian
loop measure is defined by
\begin{equation}
    \mu_{\text{loop}}=\int_{\mathbb{R}^d}dx\int_0^\infty \frac{dt}{t}\frac{\Pb_{x,x,t}}{(2\pi t) ^{\frac{d}{2}}},
\end{equation}
where  $\Pb_{x,x,t}$ is the law of Brownian bridge from $x$ to $x$ of duration $t>0$. The Brownian loop measure in a domain $D$ is given by restriction of the Brownian loop measure to loops which remain in $D$:
\begin{equation}
    \mu_{\text{loop}}^{D}(\mathrm{d}\gamma)= \mathbf{1}_{\gamma \subset D} \mu_{\text{loop}}(\mathrm{d}\gamma).
\end{equation}
The BLS $\mathcal{L}_D^\alpha$ with intensity $\alpha$ is a random
countable collection of loops in $D$, generated by a Poisson point process with
intensity measure $\alpha \mu^D_{\text{loop}}$. Basic properties of Poisson
point processes imply the following FKG inequality \cite[Lemma~2.1]{Janson84}.
	\begin{lemma}[FKG inequality]\label{lem:FKG}
		A function $f$ on the space of loop configurations is said to be \emph{increasing} if, for any realizations $\Lc'\subseteq\Lc''$ of the BLS, we have $f(\Lc') \le f(\Lc'')$. Then, for any two increasing functions $f$ and $g$,
		we have
		\begin{align*}
		\mathbb{E}[fg]\ge \mathbb{E}[f] \mathbb{E}[g].
		\end{align*}
	\end{lemma}

The following lemma controls the size of clusters in a subcritical BLS.

\begin{lemma}\label{lem:cluster_small}
Fix $\alpha_*\in I_0$.  For all $\delta>0$ and $\alpha \in [0, \alpha_*]$, the probability that every cluster in $\Lc^\alpha_\Db$ has diameter at most $\delta$ is larger than some $C(\alpha_*,\delta)>0$, where $\Lc^\alpha_\Db$ is the subcritical Brownian loop soup in 
the unit ball and $C(\alpha_*,\delta)>0$ is a constant depending on $\alpha_*,\delta$ only. 
\end{lemma}
\begin{proof}
    Since the probability is decreasing in $\alpha$, it suffices to prove
    the result for $\alpha=\alpha_*$. Because $\alpha_*$ is subcritical, the
    probability that there exists a cluster crossing the annulus
    $\mathcal{B}_r \setminus \mathcal{B}_1$ tends to zero as $r \to \infty$.
    Hence there exists $r_0>1$ such that, with probability at least
    $C'(\alpha_*)>0$, no cluster crosses the annulus
    $\mathcal{B}_{r_0} \setminus \mathcal{B}_1$. The same argument as in
    \cite[Lemma~2.17]{GLQ26} then completes the proof.
\end{proof}
\begin{lemma}\label{lem:cluster_small 1}
    Fix $\alpha_*$ in $I_0$. Let $\Lc^{\alpha}$ be the subcritical Brownian loop soup in $\mathbb{R}^3$. 
    For all $\delta> 0$ and $\alpha \in [0, \alpha_*]$, the probability that every cluster intersecting the unit sphere $S_0$ has diameter at most $\delta$ is larger than some $C(\alpha_*,\delta)>0$, where $C(\alpha_*,\delta)>0$ is a constant depending on $\alpha_*,\delta$ only. 
\end{lemma}
\begin{proof}
    As in the proof of Lemma~\ref{lem:cluster_small}, we can choose
    $r_0$ so that, with probability at least $C'(\alpha_*)>0$, no cluster
    crosses the annulus $\mathcal{B}_{r_0} \setminus \mathcal{B}_1$. Moreover,
    the unit sphere can be covered by finitely many annuli such that any
    cluster intersecting $S_0$ and having diameter larger than $\delta$ must
    cross at least one of them. The FKG inequality then completes the proof.
\end{proof}
\subsection{Brownian motions and path decompositions}
In this subsection, we review some basic facts about Brownian motion and
Brownian path decompositions, following \cite{HLLS22} and \cite{JL26}. It is
useful to view Brownian motion as a measure on paths. We write
\begin{equation}
    \mu_{x,y,t}= p_t(x,y) \Pb_{x,y,t},
\end{equation}
where $p_t(x,y)=(2\pi t)^{-\frac{3}{2}}e^{-\frac{|y-x|^2}{2t}}$ is the heat kernel.
Now we define the Brownian path measure as the following:
$$\mu_{x,y}=\int_{0}^{\infty} \mu_{x,y,t} dt.$$
Moreover, $\mu_{x,y}$ is a finite measure with total mass
$$G(x,y)=\int_0^{\infty} p_{t}(x,y)dt =\frac{1}{2\pi|x-y|}.$$ For $D\subset \mathbb{R}^{d}$ and $x,y \in D$, let $\mu_{x,y}^{D}$ be the restriction of $\mu_{x,y}$ to the curves that remain in $D$. In the remainder of the paper, we assume that $D$ has a piecewise smooth boundary.

We define the interior-to-boundary and boundary-to-boundary measures as
limits of the measures above under appropriate rescaling. For $x\in D$ and
$y\in \partial D$, define
$$\mu_{x,y}^D= \lim_{\epsilon\to0} \frac{\mu_{x,y+\epsilon n_y}}{2\epsilon},$$
where $n_y$ denotes the inward unit normal at $y$ into $D$. Similarly, the
boundary-to-boundary measure is defined by
$$\mu_{x,y}^D= \lim_{\epsilon\to0} \frac{\mu_{x+\epsilon n_x,y+\epsilon n_y}}{2\epsilon^2}.$$
We next define the bubble measure. For $x\in \partial D$, let
\begin{equation}\label{bubble measure}
    \mu^{{\rm bub},D}_x = \lim_{y\to x, y\in\partial D} \mu_{x,y}^{D}.
\end{equation}

We now define the probability measure $\mu_{0,r}^{\#}$ on Brownian
excursions between $S_0$ and $S_r$ in the annulus
$\mathcal{A}(0,r):=\mathcal{B}_{e^r}\setminus \mathcal{B}_1$. Set
\[
\mu_{0,r}=\int_{S_0}\int_{S_r}\mu^{\mathcal{A}(0,r)}_{x,y}\, dx \,  dy.
\]
Let $\mu_{0,r}^{\#}:=\mu_{0,r}/|\mu_{0,r}|$ be the corresponding
normalized probability measure, which we call the Brownian excursion measure.

We will use several tools from Brownian path decomposition. The next
lemma can be found in \cite[Proposition~2.2]{HLLS22}. Let $B_t$ be a Brownian
motion with $B_0 \ne 0$, and define
$$Y_t = \frac{B_{r(t)}}{|B_{r(t)}|^2}, \quad where 
\quad \int_{0}^{r(t)}\frac{ds}{|B_s|^2}=t.$$
\begin{lemma}\label{lem:inversion invariance}
    Let $B_t$ be a standard Brownian motion with $0< |B_0| < e ^k$. Then the distribution of 
    $$Y_s ,\quad 0 \le s \le r^{-1}(T_k),$$
    is the same as Brownian motion starting at $\frac{B_0}{|B_0|^2}$, stopped at $T_{-k}$, conditioned on $T_{-k}<\infty$.
\end{lemma}
\begin{lemma}\label{lem:glueing}
    Let $W_1= A_{\frac{1}{20}} \cap \mathcal{A}(0,2)$ and $W_2= A_{\frac{1}{18}} \cap \mathcal{A}(-\frac{1}{10},\frac{21}{10})$ be two wedges with $W_1 \subset W_2$. There exists a universal constant $C$ such that for any $x_1 \in W_1 \cap S_0$, $x_2 \in W_1 \cap S_2$,
    \begin{equation}\label{eq:lem2.5conc}
    \Pb^{x_1}(B[0, \sigma_2] \in W_2 \mid B(\sigma_2)=x_2) \ge C,    
    \end{equation}
    where $\sigma_2$ is the last time $B$ visits $S_2$.
\end{lemma}
\begin{proof}
    Note that the probability above is nothing but the ratio of Green's function, and in dimension three, there exist universal constants $C_1,C_2$ such that the Green's function $G_{W_1}(x_1,x_2) \ge C_1$ and $G_{W_2}(x_1,x_2)\le C_2$, we have that $\mbox{LHS of \eqref{eq:lem2.5conc}}\geq \frac{C_1}{C_2},$ concluding the proof.
\end{proof}
The next two lemmas are similar. They show that one can force a Brownian
motion, together with the BLS clusters it intersects, to stay in a cone. Let
$n(\epsilon)=\lceil \log_2(e/\epsilon) \rceil$. Write $n(\epsilon)=n$ for simplicity. The proof is based on a
multiscale analysis: the relevant range is divided into $n$ scales,
and at each scale there is a uniformly positive probability that the Brownian
motion, together with the BLS clusters it intersects, remains in the cone. This
yields a probability that decays only polynomially in $\epsilon$.
\begin{lemma}\label{lem:BM stays in a cone}
    There exist constants $c_1,c_2$ such that the following. For every $0<\epsilon<\frac{1}{2}$, define
    $$V^-=((u+A_{\frac{1}{20}})\cup \mathcal{B}_{\frac{\epsilon}{2}}(u))\cap \mathcal{B}_e, V=((u+A)\cup \mathcal{B}_{\epsilon}(u))\cap \mathcal{B}_e.$$
    Then if $x \in S_0,|x-u| \le \frac{\epsilon}{4}$,
\begin{equation}\label{eq:lem2.6conc}
    \Pb^{x}(B[0,T_1] \subseteq V^-) \ge c_1\epsilon^{c_2}.    
    \end{equation}
\end{lemma}
\begin{proof}
     We use a multiscale analysis. Let $V'=u+A_{\frac{1}{40}}$. Consider the concentric balls
     $$\mathcal{B}_{2^{i-1}\epsilon}(u), \quad i=0,1,..., n+1.$$
     There exists a universal positive constant $C$ such that each of the following events has probability at least $C$:
     \begin{enumerate}
         \item A Brownian motion started from $x$ stays in $\mathcal{B}_{\frac{3}{8}\epsilon}(u) \cup V'$ until hitting $\partial \mathcal{B}_{\frac{1}{2}\epsilon}$.
         \item A Brownian motion started from any point of $V' \cap \partial \mathcal{B}_{2^{i-1}\epsilon}(u)$ stays in $V^{-}$ until hitting $\partial\mathcal{B}_{2^{i}\epsilon}(u)$, and the hitting point lies in $V'$.
     \end{enumerate}
     By the strong Markov property,
     $\Pb^{x}(B[0,T_1] \subseteq V^-) \ge C^{n+2}$. 
     Hence there exist $c_1,c_2>0$ such that \eqref{eq:lem2.6conc} holds.
\end{proof}
\begin{lemma}\label{lem:loop-intersects-cone-small}
Let $\alpha_*\in I_0$. Then there exist constants $c_1,c_2>0$, depending only on
$\alpha_*$, such that for every $\alpha\in(0,\alpha_*]$, every $0<\epsilon<\frac12$,
and every pair $V,V^-$ as in Lemma~\ref{lem:BM stays in a cone}, we have
\begin{equation}\label{eq:lem2.7conc}
    \Pb\bigl(\Lambda(V^-)\subseteq V\bigr)\ge c_1\epsilon^{c_2}.
\end{equation}
\end{lemma}

\begin{proof}
    By monotonicity, it suffices to prove the result for
    $\alpha=\alpha_*$. We use the same multiscale analysis as above. Let
    $E_i$ be the event that every cluster of $\Lc^{\alpha_*}$ intersecting
    $\mathcal{B}_{2^{i-1}\epsilon}(u)$ has diameter at most
    $\frac{2^{i-1}}{100}\epsilon$. If
    $\bigcap_{i=1}^{n+1} E_i$ occurs, then
    $\Lambda(V^-) \subseteq V$. By Lemma~\ref{lem:cluster_small 1}, there
    exists a constant $C>0$ such that $\Pb(E_i)\ge C$ for every $i$. The FKG
    inequality gives
    $$\Pb(\Lambda(V^-) \subseteq V) \ge \Pb(\bigcap\limits_{i=1}^{n+1} E_i) \ge C^{\log_2{\frac{e}{\epsilon}+2}}.$$
    Hence there exist $c_1,c_2>0$ such that \eqref{eq:lem2.7conc} holds.
\end{proof}
To handle the extra difficulty caused by the Brownian loop soup, we need the
following boundary separation estimate.

\begin{lemma}\label{lem:conditioned-exit-cap}
    There exists a universal constant \(c>0\) such that for any compact set
    \(K\subseteq \Db\setminus \overline{\Db(u)}\),
    \begin{equation}\label{eq:lem2.8}
        \Pb^0\left(
        \operatorname{dist}\bigl(B(T_0),\overline{\Db\setminus \Db(u)}\bigr)
        \ge \frac12
        \,\middle|\,
        B[0,T_0]\cap K=\emptyset
        \right)\ge c .
    \end{equation}
\end{lemma}

\begin{proof}
This is the continuum analogue of \cite[Claim~3.4]{SS18}. The proof
there uses only comparison estimates for harmonic functions and the boundary
Harnack principle, and therefore applies in the present setting. For the convenience of the reader, we briefly sketch the proof.

Choose \(\delta>0\) small enough so that
\[
    M_\delta:=\mathcal B_\delta(u)\cap S_0
    \subseteq
    \left\{
    z\in S_0:
    \operatorname{dist}\bigl(z,\overline{\Db\setminus\Db(u)}\bigr)
    \ge \frac12
    \right\}.
\]
It suffices to prove a uniform lower bound for the conditional probability of
exiting through \(M_\delta\). Let
\[
    h(z)=\Pb^z(B(T_0)\in M_\delta), \qquad z\in\Db .
\]
We use the harmonic-measure estimates from \cite[Claim~3.4]{SS18}. The proof
there gives, in the continuum setting as well, a radius \(\rho\in(0,1)\)
depending only on \(\delta\), such that
\[
    h(z)\le h(0)
    \quad\text{for }z\in \Db\setminus\Db(u),
\]
and
\[
    h(z)\le \frac14 h(0)
    \quad\text{whenever } \rho\le |z|<1
    \text{ and } |z-u|\ge \frac12 .
\]
Here \(\rho\) is only an auxiliary radius: one stops the Brownian
motion on \(\partial\mathcal B_\rho\) before it exits \(\Db\), and the second
estimate says that, from points on this stopping surface which are not directed
towards the cap near \(u\), the harmonic measure of \(M_\delta\) is uniformly
smaller than from the origin.

Applying the same optional-stopping and strong Markov argument as in
\cite[Claim~3.4]{SS18} to the harmonic function \(h\), with the obstacle
\(K\subseteq\Db\setminus\overline{\Db(u)}\), yields
\[
    \Pb^0\bigl(B(T_0)\in M_\delta
    \mid B[0,T_0]\cap K=\emptyset\bigr)\ge c
\]
for a universal constant \(c>0\). Since \(M_\delta\) is contained in the target
set in \eqref{eq:lem2.8}, the lemma follows.
\end{proof}

We also refer to \cite[Appendix A]{lawler2020infinite} for a unified approach to such estimates in any dimension.

\section{Up-to-constants estimate}
In this section, we prove the up-to-constants estimate for the
generalized intersection exponent $\xi_{\alpha}(k,\lambda)$. For simplicity,
we first prove Theorem~\ref{thm:up-to-constant} in the case $k=1$. A crucial
ingredient is the separation lemma proved in the next subsection.
\subsection{Separation lemma for 3D BLS}\label{subsec:sl}
In this subsection, we prove a separation lemma for two Brownian motions
in the presence of a BLS. Informally, such a lemma says that, conditioned on
two random sets being disjoint, there is a uniformly positive probability that
they are well separated. An important ingredient is the notion of quality,
which measures the conditional probability that one Brownian motion crosses an
annulus without intersecting the other object.

We now describe the rescaled setup used in the separation argument.
Recall that \(S_r\) is the sphere of radius \(e^r\), and that
\(\mathcal B_r(x)\) is the ball of radius \(r\) centered at \(x\). Let
\(u=(1,0,0)\) and
\[
    A_\delta=\{x:x/|x|\in \mathcal B_\delta(u)\}.
\]
Let \(P\subseteq\mathcal B_1\) be a given set satisfying
\(S_0\cap P\ne\emptyset\), and let \(z\in S_0\cap P\). We call
\((P,z)\) an initial configuration, where \(z\) is the starting
point of \(B^1\). Let \(B^1\) be a Brownian motion started from \(z\). The
additional Brownian motion \(B\) is started from a point chosen uniformly on
\(S_{-\rho}\), for some \(\rho>0\). Here \(\rho\) is an auxiliary parameter
coming from the rescaling of the past Brownian trajectory in the multiscale
argument, and all estimates below are uniform in \(\rho\). Throughout, we only consider configurations for which the conditioning event
\(\{\overline B_0\cap\Lambda_0=\emptyset\}\) has positive probability. Let
\(\Gamma_r\) be the union of \(P\) and \(\overline{B^1_r}\), and let
\(\Lambda_r\) be the union of \(\Gamma_r\) and all clusters in
\(\Lc_r\setminus\Lc_0\) that intersect it.

Let $A= A_\frac{1}{10}$. We study the non-intersection probability at
scale $r$, conditioned on non-intersection at scale $0$. Define
$X_r= \Pb(\overline{B_r} \cap \Lambda_r = \emptyset \mid \overline{B_0} \cap \Lambda_0 = \emptyset ,\mathcal{F}_r)$.
Define the separation event
$$U(s,r)= B^1[T^1_{s},T_{r}^1] \subseteq -A.$$
For $r<s$, we also consider the separated non-intersection probability
$$ \overline{X}_{r,s}= \Pb((\overline{B_s} \cap \Lambda_s = \emptyset) \cap (B[T_r,T_s] \subseteq A) \mid \overline{B_0} \cap \Lambda_0 = \emptyset ,\mathcal{F}_r).$$
\begin{proposition}[Separation Lemma]\label{separation lemma}
    Fix $\alpha_*\in I_0$, $0<\lambda_0<\lambda_1$. For any $r\ge1$, any intensity $\alpha \in (0,\alpha_*]$, any $\lambda \in [\lambda_0,\lambda_1]$ and any initial configuration, there exists a constant $c=c(\lambda_0,\lambda_1,\alpha_*)$ such that
    $$\mathbb{E}\left[\overline{X}_{r-\frac{1}{2},r}^{\lambda}\mathbf{1}_{U(r-\frac{1}{2},r)}\right] \ge c \mathbb{E}[X_r^{\lambda}].$$
\end{proposition}
We follow the strategy of \cite{Lawler98}. Define
$$\delta_r = e^{-r} \min\{{\rm dist}(B(T_r),\overline{B^1_r} ), {\rm dist}(B^1(T_r^1), \overline{B_r})\}.$$
We define the {\it quality} of non-intersection by
\begin{equation}\label{eq:quality}
    Q^{r}_{\epsilon}= \Pb(\delta_r \ge \epsilon \mid \overline{B_r} \cap \Lambda_r = \emptyset, \mathcal{F}_r),
\end{equation}
and let $Q_{\epsilon}=Q_{\epsilon}^{0}$. The proof of Proposition~\ref{separation lemma} is reduced to the following two lemmas.
\begin{lemma}\label{decay rate}
     Fix $\alpha_*\in I_0$, $\lambda_1>0$. There exist  constants $c_1$,$\beta$,$\beta_0$ depending only on $\alpha_*, \lambda_1$ such that for  any intensity $\alpha \in (0,\alpha_*]$, any $\lambda \in [0,\lambda_1]$ and any initial configuration satisfying $Q_{\epsilon}\ge \epsilon$, then
    $$\mathbb{E}[\overline{X}^{\lambda}_{\frac{1}{4},1}\mathbf{1}_{U(\frac{1}{4},1)}]\ge c_1 \epsilon^{\beta \lambda+\beta_0 }.$$
\end{lemma}
\begin{lemma}\label{separate at one scale}
    Fix $\alpha_*\in I_0$, $0<\lambda_0<\lambda_1$. There exists a constant $c$ depending only on $\alpha_*,\lambda_0,\lambda_1$ such that for any initial configuration, any intensity $\alpha \in (0,\alpha_*]$, any $\lambda \in [\lambda_0,\lambda_1]$, we have
    $$\mathbb{E}[\overline{X}^{\lambda}_{\frac{1}{2},1} \mathbf{1}_{U(\frac{1}{2},1)}] \ge c \mathbb{E}[X_1^\lambda].$$
\end{lemma}
Before proving these two lemmas, we comment on the strategy. Lemma~\ref{decay rate}
is used only in the proof of Lemma~\ref{separate at one scale}; the additional
complexity caused by the BLS appears in Lemma~\ref{decay rate} only. The proof of
Lemma~\ref{separate at one scale} relies on a multiscale analysis. Intuitively,
one decomposes a scale into ``micro-scales'' according to the quality. At each
micro-scale, the quality becomes sufficiently good with high probability, and
this implies that separation occurs soon thereafter. If the corresponding
probability losses are summable over the micro-scales, one obtains the desired
uniform constant.
\begin{proof}[Proof of Lemma \ref{decay rate}]
    Cover the unit sphere by $O(\epsilon^{-2})$ balls of radius
    $\epsilon/16$ centered on the sphere. By the definition of $Q_{\epsilon}$,
    there exists $x$ such that
    $$\Pb((\delta_0 \ge \epsilon ) \cap (B(T_0) \in \mathcal{B}_{\frac{\epsilon}{16}}(x) ) \mid \overline{B_0} \cap \Lambda_0 = \emptyset) \ge c\epsilon^{2} \epsilon.$$
    By rotation invariance, we may assume $x=u$. Define
    $$V^-=(u+A_{\frac{1}{20}})\cup \mathcal{B}_{\frac{\epsilon}{2}}(u);\quad V=(u+A)\cup \mathcal{B}_{\epsilon}(u).$$
    Let 
    $$E_1= B[T_0,T_1] \subseteq V^-; \quad E_2= \Lambda(V^-) \subseteq V; \quad E_3= (B^1[T_0^1,T^1_1] \cap V =\emptyset) \cap U(\frac{1}{4},1).$$
    By an argument similar to that in Lemma~\ref{lem:BM stays in a cone},
    there exist universal constants $c_1,c_2$ such that
    \begin{equation}\label{eq:E_3}
        \Pb(E_3) \ge c_1 \epsilon^{c_2}
    \end{equation}
    Let $N=\{\overline{B_1} \cap \Lambda_1 = \emptyset\}$. We now bound
    $\mathbb{E}[\overline{X}^{\lambda}_{\frac{1}{4},1}\mathbf{1}_{U(\frac{1}{4},1)}]$
    from below.
    \begin{align*}
        &\mathbb{E}[\overline{X}^{\lambda}_{\frac{1}{4},1}\mathbf{1}_{U(\frac{1}{4},1)}]  \ge \mathbb{E}[\overline{X}^{\lambda}_{\frac{1}{4},1}\mathbf{1}_{E_2}\mathbf{1}_{E_3}] \\
         \ge \,\;&\mathbb{E}[(\mathbb{P}((\delta_0 \ge \epsilon ) \cap (B(T_0) \in \mathcal{B}_{\frac{\epsilon}{16}}(x)) \cap E_1  \mid \overline{B_0} \cap \Lambda_0 = \emptyset))^{\lambda}\mathbf{1}_{E_2}\mathbf{1}_{E_3}] \\
        \overset{\operatorname{FKG}}{\ge}&c\epsilon^{3\lambda}\Pb(E_1 \cap N \mid \overline{B_0} \cap \Lambda_0 = \emptyset)^{\lambda}  \Pb(E_2)\Pb(E_3) \overset{\eqref{eq:lem2.6conc},\eqref{eq:lem2.8}} {\underset{(*)}{\ge}} c\epsilon^{(3+\beta_1)\lambda} \Pb(E_2)\Pb(E_3) \overset{\eqref{eq:lem2.7conc},\eqref{eq:E_3}}{\ge} c \epsilon^{\beta \lambda+\beta_0}.
    \end{align*}
    We explain how \eqref{eq:lem2.6conc} and \eqref{eq:lem2.8} imply $(*)$. The obstacle is that, although the two Brownian motions
    are separated, the initial configuration $P$ can approach $B(T_0)$
    arbitrarily closely. Since $P\subseteq \Db$, conditioned on
    $\overline{B_0} \cap \Lambda_0 = \emptyset$, we can use
    \eqref{eq:lem2.8} to separate $B(t)$ from $P$. Once they are separated, we
    use \eqref{eq:lem2.6conc} to force the Brownian motion to stay in a cone,
    which gives the desired bound.
\end{proof}
We briefly explain the idea of the proof of Lemma~\ref{separate at one scale}. Ideally, one would like to
decompose the scale into dyadic subscales. However, in order to apply the
separation argument, one has to allow some extra time for the quality to
improve. To handle this, we associate to the $n$-th scale a time interval
of length $n^2 2^{-n}$. We then show that, at each such micro-scale, the
corresponding loss is summable in $n$. This yields a uniform positive
probability of separation.
\begin{proof}[Proof of Lemma \ref{separate at one scale}]
    Choose $N$ sufficiently large so that
    $$\sum_{n=N+1}^{\infty} \frac{n^2}{2^{n}} \le \frac{1}{4}.$$
    Let $a_{n+1}=a_n+\frac{n^2}{2^{n}}$ for $n\ge N$, and set
    $a_n=\frac{1}{4}$ for $n \le N$. Then $a_n \le \frac{1}{2}$ for all $n$.
    Recall the definition of $Q_{\epsilon}$ from \eqref{eq:quality}. For
    simplicity, write $Q^n=Q_{2^{-n}}$.
    Let 
    $$r_n= \inf \frac{\mathbb{E}[\overline{X}_{a_n,1}^{\lambda}\mathbf{1}_{U(a_n,1)}]}{\mathbb{E}[X^{\lambda}_{a_n}]}.$$
    Here the infimum is over all initial configurations satisfying
    $Q^{n+2} \ge 1/2$.
    The numerator in the ratio defining $r_n$ is increasing in $a_n$,
    while the denominator is decreasing in $a_n$. Hence
    $$\mathbb{E}[\overline{X}_{\frac{1}{2},1}^{\lambda}\mathbf{1}_{U(\frac{1}{2},1)}]\ge \inf\limits_{n} r_n \mathbb{E}[X_1^{\lambda}].$$
    We will show that there exist $\epsilon_n$ such that
    $r_n \ge (1-\epsilon_n)r_{n-1}$ and
    $\sum_{n=1}^{\infty} \epsilon_{n}$ is uniformly bounded. This gives the
    desired constant.

    Let $\sigma_n$ denote the smallest positive integer $j$ such that
    $$\Pb(\delta_{j2^{-n}} \ge 2^{-n-1} \mid B(0,T_{j2^{-n}}] \cap \Lambda_{j2^{-n}} =\emptyset) \ge \frac{1}{2}.$$
    Assume that the initial configuration satisfies $Q^{n+2}\ge\frac{1}{2}$.
    If $j<n^2$, on the event $\sigma _n=j$ we have 
    $$\mathbb{E}[\overline{X}_{a_n,1}^{\lambda}\mathbf{1}_{U(a_n,1)} \mid \mathcal{F}_{j2^{-n}}]= X^{\lambda}_{j2^{-n}}\mathbb{E}\Big[\frac{\overline{X}_{a_n,1}^{\lambda}}{X^{\lambda}_{j2^{-n}}} \mathbf{1}_{U(a_n,1)} \mid \mathcal{F}_{j2^{-n}}\Big].$$
    $$\mathbb{E}[{X}_{a_n}^{\lambda} \mid \mathcal{F}_{j2^{-n}}] =X^{\lambda}_{j2^{-n}} \mathbb{E}\Big[\frac{X_{a_n}^{\lambda}}{X^{\lambda}_{j2^{-n}}}  \mid \mathcal{F}_{j2^{-n}}\Big].$$
    Dividing the first equation by the second gives
    \begin{equation*}
        \frac{\mathbb{E}[\overline{X}_{a_n,1}^{\lambda}\mathbf{1}_{U(a_n,1)} \mid \mathcal{F}_{j2^{-n}}]}{\mathbb{E}[{X}_{a_n}^{\lambda} \mid \mathcal{F}_{j2^{-n}}]}=\frac{\mathbb{E}\Big[\frac{\overline{X}_{a_n,1}^{\lambda}}{X^{\lambda}_{j2^{-n}}} \mathbf{1}_{U(a_n,1)} \mid \mathcal{F}_{j2^{-n}}\Big]}{\mathbb{E}\Big[\frac{X_{a_n}^{\lambda}}{X^{\lambda}_{j2^{-n}}}  \mid \mathcal{F}_{j2^{-n}}\Big]}
    \end{equation*}
    By the scaling invariance of Brownian motion and the BLS, and by the
    definition of $r_n$, on the event $\sigma_n=j$ we have
    $$\mathbb{E}[\overline{X}_{a_n,1}^{\lambda}\mathbf{1}_{U(a_n,1)} \mathbf{1}_{\sigma_n=j} \mid \mathcal{F}_{j2^{-n}}] \ge r_{n-1} \mathbb{E}[{X}_{a_n}^{\lambda} \mathbf{1}_{\sigma_n=j} \mid \mathcal{F}_{j2^{-n}}].$$
    Summing over all $j\le n^2$ gives
     $$\mathbb{E}[\overline{X}_{a_n,1}^{\lambda}\mathbf{1}_{U(a_n,1)} \mathbf{1}_{\sigma_n \le n^2} ] \ge r_{n-1} \mathbb{E}[{X}_{a_n}^{\lambda} \mathbf{1}_{\sigma_n \le n^2}].$$
     We now estimate the weighted contribution of the event
     $\{\sigma_n > n^2\}$. Let $Q(1)$ denote the random variable
     $\Pb(\delta_{2^{-n}} \ge 2^{-n-1} \mid B(0,T_{2^{-n}}] \cap \Lambda_{2^{-n}} =\emptyset)$.
     We claim that there exists a constant $c_2(\lambda_0)>0$ such that, for
     every $n\ge1$ and every initial configuration with
     $Q^{n+1} \le \frac{1}{2}$,
     we have
     \begin{equation}\label{eq:intersect with positive probability}
         \mathbb{E}[X_{2^{-n}}^{\lambda} \mathbf{1}_{Q(1)<\frac{1}{2}}] \le 1-c_2.
     \end{equation}
     An elementary inequality gives
     $\mathbb{E}[1-X_{2^{-n}}^{\lambda}]\ge
     \min \{\lambda_0,1 \}\mathbb{E}[1-X_{2^{-n}}]$. It therefore suffices to
     show that, for any initial configuration with $Q^{n+1}\le \frac{1}{2}$,
     we have $\mathbb{E}[1-X_{2^{-n}}]\ge c_2'$ for some universal constant
     $c_2'>0$. This term is bounded from below by the intersection probability
     of two Brownian motions, and the desired bound follows from Brownian
     scaling invariance.
     By iterating \eqref{eq:intersect with positive probability}, we obtain
     \[
        \mathbb{E}[{X}_{a_n}^{\lambda} \mathbf{1}_{\sigma_n > n^2}] \le (1-c_2)^{n^2}=e^{-\beta_1 n^2},
     \]
     for some $\beta_1>0$.
     Using Lemma~\ref{decay rate}, we conclude that
     $$ \mathbb{E}[{X}_{a_n}^{\lambda} \mathbf{1}_{\sigma_n \le n^2}] \ge (1-c_1\exp(-\beta_1 n^2 + \beta\lambda + \beta_0)) \mathbb{E}[{X}_{a_n}^{\lambda}].$$
     Since the sum of
     $c_1\exp(-\beta_1 n^2 + (\beta\lambda +\beta_0)n)$ is uniformly bounded
     for $\lambda \in [\lambda_0,\lambda_1]$, the lemma follows.
\end{proof}
We now use Lemma~\ref{separate at one scale} to prove Proposition~\ref{separation lemma}.
\begin{proof}[Proof of Proposition~\ref{separation lemma}]
    Set the initial configuration
    $P=e^{-(r-1)}\Lambda_{r-1}$ and
    $z=e^{-(r-1)}B^1(T^1_{r-1})$. By scaling invariance,
$$\mathbb{E}\left[\frac{\overline{X}^{\lambda}_{r-\frac{1}{2},r}}{X_{r-1}^{\lambda}} \mathbf{1}_{U(r-\frac{1}{2},r)}\Bigg| \mathcal{F}_{r-1}\right] \ge c \mathbb{E}\left[\frac{X_{r}^\lambda}{X_{r-1}^{\lambda}} \Big| \mathcal{F}_{r-1}\right].$$
This gives Proposition~\ref{separation lemma}.
\end{proof}

\subsection{Up-to-constants estimates for {$p(\alpha,1,r,\lambda)$}}
In this subsection, we prove Theorem~\ref{thm:up-to-constant} for $k=1$ (the up-to-constants estimates for the generalized non-intersection probability) using the separation lemma from the
previous subsection. For simplicity, let $\phi(r)=p(\alpha,1,r,\lambda)$. For
$r<s$, define $\Lambda_{r,s}$ to be the Brownian excursion
$B^1[\sigma^1_{r},T^1_{s}]$ together with the clusters in
$\Lc_s \setminus \Lc_r$ that it intersects. Let $\overline{N}_{r,s}$ denote
the event $B[0,T_r] \cap \Lambda_s = \emptyset$, and let $N_{r,s}$ be the
event $B[\sigma_{r},T_{s}] \cap \Lambda_{r,s} =\emptyset$. Here
\(\sigma_r\) and \(\sigma_r^1\) denote the last hitting times of \(S_r\)
before \(T_s\) and \(T_s^1\), respectively.

\begin{proof}[Proof of Theorem~\ref{thm:up-to-constant} for $k=1$]
    For the lower bound, it suffices to show that  
    \begin{equation}\label{eq:sub}
        \phi(r+s+1) \le \phi(r) \phi(s).
    \end{equation}
    The proof is the same decoupling argument as in
\cite[Theorem~3.3]{GLQ26}. Although that result is stated in two dimensions,
the argument uses only the Markov property, scale invariance, and independence
of the loop soup in disjoint annuli, and therefore applies verbatim in the
present three-dimensional setting.

    For the upper bound, it suffices to show that there exists a universal constant $C_2>0$ such that
    $$\phi(r+s+2) \ge C_2 \phi(r)\phi(s).$$
    Let $E$ be the event $B[\sigma_{r+2},\sigma_{r+2+\frac{1}{2}}] \in A$, $E_1$ be the event $B[\sigma^1_{r+2},\sigma^1_{r+2+\frac{1}{2}}] \in -A$.  Let $F$ be the event $B[T_{r},T_{r+\frac{1}{2}}] \in A$, $F_1$ be the event $B^1[T^1_{r},T^1_{r+\frac{1}{2}}] \in -A$. Let $L$ be the event that all clusters in $\Lc$ that intersect $\mathcal{A}(r-\frac{1}{10},r+\frac{21}{10})$ are of diameter less than $\frac{e^r}{100}$. Let $$U:=\{B[T_r,\sigma_{r+2}] \in A_{\frac{1}{18}} \cap \mathcal{A}(r-\frac{1}{10},r+\frac{21}{10})\};\quad U_1:=\{B^1[T^1_r,\sigma^1_{r+2}] \in -A_{\frac{1}{18}} \cap \mathcal{A}(r-\frac{1}{10},r+\frac{21}{10})\}.$$
    By Proposition~\ref{separation lemma},
    $$\mathbb{E}[(\Pb(\overline{N}_{r,r}\cap F))^{\lambda}\mathbf{1}_{F_1}] \ge c \mathbb{E}[X_r^{\lambda}].$$
    By the inversion invariance of Brownian motion,
    $$\mathbb{E}[(\Pb(N_{r+2,r+s+2}\cap E))^{\lambda}\mathbf{1}_{E_1}] \ge c \mathbb{E}[X_s^{\lambda}].$$
   Hence,
\[
\phi(r+s+2)\ge
\mathbb{E}\!\left[
\Pb(\overline{N}_{r,r+s+2}\cap F\cap N_{r+2,r+s+2}\cap E\cap U
\mid \mathcal{F}_{r+s+2})^{\lambda}
\mathbf{1}_{F_1}\mathbf{1}_{E_1}\mathbf{1}_{U_1}\mathbf{1}_{L}
\right].
\]
By the Markov decomposition at the two boundary spheres, the middle piece is a
Brownian excursion whose endpoints are uniformly distributed on the two spheres;
conditioned on these endpoints, it is independent of the inner and outer pieces.
Therefore Lemma~\ref{lem:glueing} gives a uniform positive lower bound for the
event \(U\). On the event \(U\cap U_1\cap L\), the middle annulus does not create
an intersection between the inner and outer parts. Thus, for some \(C_3>0\),
\[
\phi(r+s+2)\ge
C_3\mathbb{E}\!\left[
\Pb(\overline{N}_{r,r}\cap F \mid \mathcal{F}_{r+s+2})^{\lambda}
\Pb(N_{r+2,r+s+2}\cap E \mid \mathcal{F}_{r+s+2})^{\lambda}
\mathbf{1}_{F_1}\mathbf{1}_{E_1}\mathbf{1}_{U_1}\mathbf{1}_{L}
\right].
\]
Applying the same gluing estimate to the middle piece of \(B^1\), and using the
independence of the loop soups in the corresponding annuli together with the
FKG inequality for the cluster-small event \(L\), there exists \(C'>0\) such
that
\[
\phi(r+s+2)\ge
C_3^2C'\,
\mathbb{E}\!\left[
\Pb(\overline{N}_{r,r}\cap F \mid \mathcal{F}_{r+s+2})^{\lambda}\mathbf{1}_{F_1}
\right]
\mathbb{E}\!\left[
\Pb(N_{r+2,r+s+2}\cap E \mid \mathcal{F}_{r+s+2})^{\lambda}\mathbf{1}_{E_1}
\right].
\]

   Using Proposition~\ref{separation lemma} and inversion invariance, we obtain
   $$\phi(r+s+2)\ge (C_3)^2C'c^2 \phi(r) \phi(s).$$
   Choosing $C_2= (C_3)^2C'c^2$ and noting that all constants depend only
   on $\lambda_0,\lambda_1$ and $\alpha_*$ completes the proof.
\end{proof}
\subsection{Up-to-constants estimate for {$p(\alpha,k,r,\lambda)$}}
We now extend the argument from one Brownian path to a fixed number \(k\)
of Brownian paths. The only new point is that the separation lemma must
be stated with all \(k\) Brownian paths kept in the opposite cone. We give the
statement explicitly.

Let \(P\subseteq\mathcal B_1\) be a deterministic initial set and let
\(z_1,\ldots,z_k\in S_0\cap P\) be the starting points of
\(B^1,\ldots,B^k\). Set
\[
    \Gamma_r^{(k)}
    =
    P\cup \bigcup_{i=1}^k \overline{B^i_r}.
\]
Let \(\Lambda_r^{(k)}\) be the union of \(\Gamma_r^{(k)}\) and all
clusters in \(\Lc_r\setminus\Lc_0\) that intersect \(\Gamma_r^{(k)}\). Define
\[
    X_r^{(k)}
    =
    \Pb\bigl(
        \overline{B_r}\cap \Lambda_r^{(k)}=\emptyset
        \mid
        \overline{B_0}\cap \Lambda_0^{(k)}=\emptyset,\mathcal F_r
    \bigr),
\]
and, for \(r<s\),
\[
    \overline{X}_{r,s}^{(k)}
    =
    \Pb\bigl(
        \{\overline{B_s}\cap \Lambda_s^{(k)}=\emptyset\}
        \cap \{B[T_r,T_s]\subseteq A\}
        \mid
        \overline{B_0}\cap \Lambda_0^{(k)}=\emptyset,\mathcal F_r
    \bigr).
\]
The corresponding separation event is
\[
    U_k(s,r)
    =
    \bigcap_{i=1}^k
    \{B^i[T_s^i,T_r^i]\subseteq -A\}.
\]

\begin{proposition}[Separation lemma for general \(k\)]\label{prop:separation-general-k}
    Fix \(\alpha_*\in I_0\), \(k\in\mathbb N\), and
    \(0<\lambda_0<\lambda_1\). For every \(r\ge1\), every
    \(\alpha\in(0,\alpha_*]\), every \(\lambda\in[\lambda_0,\lambda_1]\), and
    every \(k\)-path initial configuration as above, there exists
    \(c=c(k,\lambda_0,\lambda_1,\alpha_*)>0\) such that
    \begin{equation}\label{eq:slk}
        \mathbb{E}\left[
            \bigl(\overline{X}_{r-\frac12,r}^{(k)}\bigr)^\lambda
            \mathbf{1}_{U_k(r-\frac12,r)}
        \right]
        \ge
        c\,\mathbb{E}\left[\bigl(X_r^{(k)}\bigr)^\lambda\right].
    \end{equation}
\end{proposition}

\begin{proof}
    The proof is the same as the proof of
    Proposition~\ref{separation lemma}, with constants allowed to depend on
    \(k\). The quality variable is replaced by
    \[
        \delta_r^{(k)}
        =
        e^{-r}
        \min\left\{
            \operatorname{dist}\left(B(T_r),\bigcup_{i=1}^k \overline{B^i_r}\right),
            \min_{1\le i\le k}
            \operatorname{dist}\left(B^i(T_r^i),\overline B_r\right)
        \right\}.
    \]
    In the analogue of Lemma~\ref{decay rate}, the Brownian motion \(B\)
    is forced into the cone \(A\), while each of the \(k\) Brownian paths is
    forced into \(-A\). This changes only the polynomial power of
    \(\epsilon\). The cluster-containment estimate is unchanged, and the
    micro-scale iteration in Lemma~\ref{separate at one scale} remains summable
    because \(k\) is fixed. Scaling then gives \eqref{eq:slk}.
\end{proof}

We now finish the proof of Theorem~\ref{thm:up-to-constant} for general $k$.
\begin{proof}[Proof of Theorem~\ref{thm:up-to-constant} for general
\(k\)]

Let \(\phi_k(r)=p(\alpha,k,r,\lambda)\). The lower bound is obtained
from the same Markov decomposition as in the case \(k=1\), giving
\[
    \phi_k(r+s+1)\le \phi_k(r)\phi_k(s).
\]
For the reverse inequality, apply
Proposition~\ref{prop:separation-general-k} to the inner pieces and, after
inversion, to the outer pieces. The gluing step is unchanged except that it is
performed for each of the \(k\) Brownian paths; this costs only a positive
constant depending on \(k\). The cluster-small event in the middle annulus again
prevents the loop soup from creating intersections between the inner and outer
pieces. Hence there exists \(C=C(k,\lambda_0,\lambda_1,\alpha_*)>0\) such that
\[
    \phi_k(r+s+2)\ge C\,\phi_k(r)\phi_k(s).
\]
Together with the preceding submultiplicative inequality, this gives
\eqref{eq:up-to-constant} for \(p(\alpha,k,r,\lambda)\), with constants
depending only on \(k\), \(\alpha_*\), \(\lambda_0\), and \(\lambda_1\).
\end{proof}

\section{Dimension of local cut points in a BLS}
In this section, we compute the Hausdorff dimension of local cut points
of the three-dimensional BLS. The main input is
Theorem~\ref{thm:up-to-constant}, proved in the previous section. We follow
the framework developed in \cite{GLQ26} (see also \cite{Lawler96}). First, we compute the dimension of
local cut points on a single Brownian loop in the presence of an independent
BLS. Second, we prove an almost sure upper bound by a first-moment estimate
and a positive-probability lower bound by a second-moment estimate. Finally, we
use ergodicity of the BLS to upgrade the lower bound to an almost sure
statement.
\begin{proposition}\label{Estimate of dimension}
    For all $\alpha \in I_0$, the following holds almost surely. 
    $${\rm dim}_{\mathcal{H}}(G_{\rm loc}) \le \max \{2-\xi_{\alpha}(1,1),0 \},$$
    and the following holds with positive probability
    $${\rm dim}_{\mathcal{H}}(G_{\rm loc}) \ge \max \{2-\xi_{\alpha}(1,1),0 \}.$$
\end{proposition}

\subsection{Local cut points on a single Brownian loop}
The purpose of this subsection is to isolate the local picture around a
single loop which is relevant for cut points. A local cut point on a loop is
detected by looking at two strands of the same loop entering and leaving a
small neighborhood of the point, together with the loop-soup clusters attached
to these strands. Thus the problem reduces to estimating the probability that
two Brownian excursions across a small annulus remain disjoint after the
loop-soup enlargement. This is precisely the type of event controlled by the
generalized non-intersection probabilities from
Theorem~\ref{thm:up-to-constant}. The notation below is introduced in order to
make this reduction uniform over all small boxes in the bulk.

Fix $r \in [1/4,1/2]$ and $a\in S_0$. Let
$D_0=[-1/16,1/16]^3$. Sample a subcritical BLS $\Lc^{\alpha}$ in the whole
space $\mathbb{R}^3$. Let $\gamma$ be a Brownian loop sampled from the measure
$\mu^{{\rm bub},\mathcal{B}_r}_{ar}$, conditioned on
$\{\gamma \cap D_0 \ne \emptyset\}$. For $x\in\gamma$ and $\epsilon>0$, let
$\mathcal U_\epsilon^\gamma(x)$ be the connected component of $\gamma \cap \mathcal{B}_{\epsilon}(x)$ that contains $x$. Let $\mathcal C_\epsilon^\gamma(x)$ be the
union of $\mathcal U_\epsilon^\gamma(x)$ together with all clusters of $\Lc_{\mathcal{B}_{\epsilon}(x)}^{\alpha{}}$ it intersects. For
$j\ge1$, we say that $x\in\gamma$ is a $j$-cut point if
$\mathcal C_{2^{-j}}^\gamma(x)\setminus\{x\}$ is no longer connected. Let
$\mathcal E_j$ be the set of $j$-cut points, and let
$G=\bigcup_{j\ge4}\mathcal E_j$ be the set of local cut points on $\gamma$ in
the independent loop soup. We prove the following proposition.
\begin{proposition}\label{dimension of a single Brownian loop}
     For all $\alpha \in I_0$, the following holds almost surely:
    \begin{equation}
        {\rm dim}_{\mathcal{H}}(G) \le \max \{2-\xi_{\alpha}(1,1),0 \},
    \end{equation}
    and the following holds with positive probability
    \begin{equation}
        {\rm dim}_{\mathcal{H}}(G) \ge \max \{2-\xi_{\alpha}(1,1),0 \}.
    \end{equation}
\end{proposition}
For each $n \ge 3$, divide $D_0$ into $2^{3(n-3)}$ non-overlapping
closed cubes of side length $2^{-n}$; we call these \(n\)-cubes. Let
$j\ge 4$ and $n \ge j+4$. Suppose that $D$ is an \(n\)-cube centered at
$v_D$. Let $t_\gamma$ be the time length of the Brownian loop $\gamma$, let
$H(D)$ be the event that $\gamma$ hits $D$, and let $\tau(D)$ denote the
hitting time of $D$.

By standard estimates on hitting probabilities of Brownian motion, there exist
universal constants $c_1,c_2>0$ such that for every such cube $D$ and every
$n\ge3$,
\begin{equation}\label{hitting probability}
    c_1 2^{-n} \le \Pb(H(D)) \le c_2 2^{-n}.
\end{equation}
Suppose \(D\) is an \(n\)-cube of side length \(2^{-n}\) centered at
\(v_D\).  We set
\[
    \widetilde D:=\mathcal B_{2^{-n}}(v_D),
    \qquad
    D^j:=\mathcal B_{2^{-j}-2^{-n}}(v_D).
\]
Then, for \(n\ge j+4\),
\[
    D\subseteq \widetilde D\subseteq D^j,
\]
and, moreover, if \(x\in D\), then \(D^j\subseteq \mathcal B_{2^{-j}}(x)\).
 Denote the annulus \(D^j\setminus\widetilde D\) by \(\mathcal A_D\), and let
\(\partial_{\rm in}:=\partial\widetilde D\) and
\(\partial_{\rm out}:=\partial D^j\) be its inner and outer boundaries,
respectively.
Let $v=\tau(D)$. Define
\begin{align*}
    s_1=\sup \{t< v : \gamma(t) \in \partial_{\rm out} \} &, \quad t_1= \inf\{t>s_1:\gamma(t)\in \partial_{\rm in}\}, \\
    s_2=\sup\{t>v:\gamma(t)\in \partial_{\rm in}\} &, \quad t_2 = \inf\{t>s_2: \gamma(t) \in \partial_{\rm out}\}.
\end{align*}
Let $\gamma^1= \gamma([s_1,t_1])$, $\gamma^2= \gamma([s_2,t_2])$, $X=\gamma([t_1,s_2])$. Let $\widetilde{\gamma}^1$ (respectively, $\widetilde{\gamma}^2$) denote the union of $\gamma^1$ (respectively, $\gamma^2$) with the clusters in $\Lc_{\mathcal{A}_D}$ that it intersects.

We say that $D$ is a $(j,n)$-cut-box if
$\widetilde{\gamma}^1\cap\widetilde{\gamma}^2=\emptyset$. Let $K_{j,n}$ be the set of
all $(j,n)$-cut-boxes. Let $F_{j,n}$ be the set of all $(j,n)$-cut-boxes
$D$ such that $\gamma$ crosses $\mathcal{A}_D$ at least four times. We will
use the following first- and second-moment estimates to prove
Proposition~\ref{dimension of a single Brownian loop}.
\begin{proposition}\label{First and second moment}
    For all $j\ge4$, there exist constants $C_1,C_2,C_3,C_4 >0$
    depending only on $r$, $j$, and $\alpha$ such that, for any
    $n\ge m \ge j+4$ and any cube $D$ of side length $2^{-n}$,
    we have
    \begin{equation}\label{first moment}
        C_1 2^{-n(1+\xi_{\alpha}(1,1))} \le \Pb(D \in K_{j,n}) \le C_2 2^{-n(1+\xi_{\alpha}(1,1))},
    \end{equation}
    and, for any pair of cubes $D,E$ of side length $2^{-n}$ with
    distance in $[2^{-m-1},2^{-m}]$,
    we have
    \begin{equation}\label{second moment}
        \Pb(D,E \in K_{j,n})\le C_3 2^{(-2n+m)(\xi_{\alpha}(1,1)+1)}.
    \end{equation}
    Furthermore, 
    \begin{equation}\label{eq:F}
        \Pb(D\in F_{j,n}) \le C_4 2^{-n(2+\xi_{\alpha}(1,1))}.
    \end{equation}
\end{proposition}
 We first prove Proposition~\ref{dimension of a single Brownian loop} assuming Proposition~\ref{First and second moment}.
\begin{proof}[Proof of Proposition~\ref{dimension of a single Brownian loop},
assuming Proposition~\ref{First and second moment}]
Fix \(j\ge4\). We first record two elementary inclusions. Almost surely,
\begin{equation}\label{eq:Ej-Kjn}
    \mathcal E_j\cap D_0
    \subseteq
    \bigcap_{n\ge j+4}\bigcup_{D\in K_{j,n}}D ,
\end{equation}
and
\begin{equation}\label{eq:Kjn-Fjn}
    \left(
    \bigcap_{n\ge j+4}\bigcup_{D\in K_{j,n}}D
    \right)\setminus \mathcal E_j
    \subseteq
    \bigcap_{n\ge j+4}\bigcup_{D\in F_{j,n}}D .
\end{equation}

Indeed, let \(x\in \mathcal E_j\cap D_0\), and let \(D\) be the
\(n\)-cube containing \(x\).  Since \(D^j\subseteq \mathcal B_{2^{-j}}(x)\),
the two strands of \(\gamma\) entering and leaving \(\widetilde D\) lie in
the two different components of
\(\mathcal C^{\gamma}_{2^{-j}}(x)\setminus\{x\}\).  Therefore their loop-soup
enlargements are disjoint, and hence \(D\in K_{j,n}\). This proves
\eqref{eq:Ej-Kjn}.

For \eqref{eq:Kjn-Fjn}, suppose that \(x\) belongs to the left-hand side. Let \(D_n\) be an \(n\)-cube containing \(x\) with
\(D_n\in K_{j,n}\). If, along an infinite sequence of \(n\)'s, the loop crossed
\(\mathcal A_{D_n}\) fewer than four times, then the two distinguished
crossings would be the only two strands connecting \(\widetilde D_n\) to
\(\partial D_n^j\). Since \(D_n\in K_{j,n}\), these two strands, together with
the loop-soup clusters attached to them inside \(\mathcal A_{D_n}\), do not intersect each other. Letting \(n\to\infty\), and using
\(D_n^j\subseteq\mathcal B_{2^{-j}}(x)\), this would imply
\(x\in\mathcal E_j\), a contradiction. Hence, for all sufficiently large
\(n\), every such \(D_n\) belongs to \(F_{j,n}\).

The following arguments are standard, see e.g. \cite{Lawler96,GLQ26}, hence we will be brief.
By the standard first-moment argument, using 
\eqref{first moment} and \eqref{eq:Ej-Kjn}, we obtain that 
\[
    \dim_{\mathcal H}(\mathcal E_j\cap D_0)
    \le
    \max\{2-\xi_\alpha(1,1),0\}
    \quad and \quad \dim_{\mathcal H}
    \left(
    \bigcap_{n\ge j+4}\bigcup_{D\in F_{j,n}}D
    \right)
    \le
    \max\{1-\xi_\alpha(1,1),0\}
    \quad\text{a.s.}
\]
The second-moment estimate (a combination of \eqref{first moment} and \eqref{second moment}) gives that, with positive probability,
\[
    \dim_{\mathcal H}
    \left(
    \bigcap_{n\ge j+4}\bigcup_{D\in K_{j,n}}D
    \right)
    \ge
    \max\{2-\xi_\alpha(1,1),0\}.
\]

Since $\max\{1-\xi_\alpha(1,1),0\} < \max\{2-\xi_\alpha(1,1),0\}$,
\eqref{eq:Kjn-Fjn} implies that, on the
positive-probability event above,
\[
    \dim_{\mathcal H}(\mathcal E_j\cap D_0)
    \ge
    2-\xi_\alpha(1,1).
\]
If \(2-\xi_\alpha(1,1)\le0\), the desired lower bound is trivial under the
convention \(\dim_{\mathcal H}(\varnothing)=0\).  Hence, for some \(j\ge4\),
\[
    \dim_{\mathcal H}(G)
    \ge
    \max\{2-\xi_\alpha(1,1),0\}
\]
with positive probability. Since \(G=\bigcup_{j\ge4}\mathcal E_j\) up to the
choice of the local radius, the almost sure upper bound follows from the
countable union over \(j\). This proves the proposition.
\end{proof}

\subsection{First and second moment estimate}
 In this subsection, we prove Proposition~\ref{First and second moment}.
 Although $\gamma^1$ and $\gamma^2$ are not independent, they are comparable to
 independent excursions in the sense of the following lemma, which can be found
 in \cite[Lemma~5.4]{GLQ26}.
\begin{lemma}\label{Translate loop into excursions} For any given n-cube $D$, on the event $H(D)$, the following holds.

    Let $Y^1,Y^2$ be two independent Brownian excursions crossing the annulus $\mathcal{A}_D$. Then the joint law of $(\gamma^1,\gamma^2)$ has a density with respect to that of $(Y^1,Y^2)$ which is uniformly bounded from $0$ and $\infty$ (the bound only depends on $j$).
\end{lemma}
We are in a position to prove Proposition~\ref{First and second moment}.
\begin{proof}[Proof of  Proposition~\ref{First and second moment}]
    We prove \eqref{first moment} first. Let $\widetilde{\gamma}^1$ be the union of $\gamma^1$ together with clusters it intersects. By Theorem \ref{thm:up-to-constant} and Lemma \ref{Translate loop into excursions}, there exist constants $c_3,c_4>0$, depending only on $j$ and $\alpha$, such that
    \begin{equation}
        c_3 2^{-n \xi_{\alpha}(1,1)}
        \le
        \Pb(D \in K_{j,n} \mid H(D))
        \le
        c_4 2^{-n \xi_{\alpha}(1,1)}.
    \end{equation}
    Combining this with \eqref{hitting probability}, there exist constants $c_5,c_6>0$, depending only on $r$, $j$, and $\alpha$, such that
    \begin{equation}
        c_5 2^{-n(1+\xi_{\alpha}(1,1))}
        \le
        \Pb(D \in K_{j,n})
        \le
        c_6 2^{-n(1+\xi_{\alpha}(1,1))}.
    \end{equation}
    This proves \eqref{first moment}.
    For \eqref{eq:F}, if \(D\in F_{j,n}\), then besides the two distinguished crossings
used in the definition of \(K_{j,n}\), the loop makes at least one additional
crossing of \(\mathcal A_D\). Conditionally on \(H(D)\) and on the two
distinguished crossings, this costs an extra factor \(C_4'2^{-n}\) for some universal constant $C_4'$. Combining this with \eqref{first moment} gives
\[
    \Pb(D\in F_{j,n})\le C_4 2^{-n(2+\xi_\alpha(1,1))}.
\]

Finally, we prove \eqref{second moment}. Let \(D,E\) be two \(n\)-cubes with distance in
\([2^{-m-1},2^{-m}]\). We use
\(\Pb(H(D)\cap H(E))\le C2^{-2n+m}\), and work on the event
\(H(D)\cap H(E)\).

\begin{description}
    \item[Case 1.] \(n\ge m+5\).
\end{description}
Let \(v_D,v_E\) be the centers of \(D,E\), and let \(x\) be their midpoint. Set
\(\mathcal A_D^m:=\mathcal B_{2^{-m-3}}(v_D)\setminus\mathcal B_{2^{-n}}(v_D)\),
\(\mathcal A_E^m:=\mathcal B_{2^{-m-3}}(v_E)\setminus\mathcal B_{2^{-n}}(v_E)\), and
\(\mathcal A_{D,E}^m:=\mathcal B_{2^{-j-1}}(x)\setminus\mathcal B_{2^{-m+2}}(x)\).
These three annuli are disjoint. Let \(\gamma^1(D),\gamma^2(D)\) be the first
and last crossings of \(\mathcal A_D^m\), and define
\(\gamma^1(E),\gamma^2(E)\), \(\gamma^1(D,E),\gamma^2(D,E)\) similarly.

For \(V\subseteq\mathcal A\), write \(\Lambda_{\mathcal A}(V)\) for \(V\)
together with the clusters of the loop soup restricted to \(\mathcal A\) that
intersect \(V\). Let
\(G_D:=\{\Lambda_{\mathcal A_D^m}(\gamma^1(D))\cap\gamma^2(D)=\emptyset\}\),
and define \(G_E,G_{D,E}\) analogously. By the same path-decomposition argument
as in Lemma~\ref{Translate loop into excursions}, the six crossings can be
replaced by independent Brownian excursions up to multiplicative constants. Since the three
annuli are disjoint, the restricted loop soups are independent. Thus
\[
\Pb(D,E\in K_{j,n}\mid H(D)\cap H(E))
\le C\Pb(G_D)\Pb(G_E)\Pb(G_{D,E})
\le C2^{-2(n-m)\xi_\alpha(1,1)}2^{-m\xi_\alpha(1,1)}.
\]
Multiplying by \(\Pb(H(D)\cap H(E))\le C2^{-2n+m}\) gives
\[
\Pb(D,E\in K_{j,n})\le C2^{-2n+m}2^{-(2n-m)\xi_\alpha(1,1)}
= C2^{(-2n+m)(1+\xi_\alpha(1,1))}.
\]

\begin{description}
    \item[Case 2.] \(n\le m+4\).
\end{description}
Let \(\gamma^1(D,E),\gamma^2(D,E)\) be the first and last crossings of
\(\mathcal A_{D,E}^m:=\mathcal B_{2^{-j-1}}(x)\setminus\mathcal B_{2^{-m+2}}(x)\).
Replacing them by independent excursions \(Y^1,Y^2\) costs only a constant, and
the loop soup is restricted to \(\mathcal A_{D,E}^m\). Hence
\[
\Pb(D,E\in K_{j,n}\mid H(D)\cap H(E))
\le C\Pb(\Lambda_{\mathcal A_{D,E}^m}(Y^1)\cap Y^2=\emptyset)
\le C2^{-n\xi_\alpha(1,1)}.
\]
Since \(\Pb(H(D)\cap H(E))\le C2^{-n}\), we get
\(\Pb(D,E\in K_{j,n})\le C2^{-n(1+\xi_\alpha(1,1))}\), which is equivalent to
\eqref{second moment} because \(n-m\in\{0,1,2,3,4\}\).
\end{proof}

\subsection{Proof of Theorem \ref{dimension}}

We are now ready to prove Proposition~\ref{Estimate of dimension} using Proposition~\ref{dimension of a single Brownian loop}.
\begin{proof}[Proof of Proposition~\ref{Estimate of dimension}]
The proof uses the same Palm-transfer argument as in
\cite[Section~5.4]{GLQ26}. We only sketch it. After restricting to loops intersecting a
fixed ball and having diameter at least \(\epsilon\), the loop soup is finite,
and the Campbell--Palm formula shows that a chosen loop together with the
remaining soup is mutually absolutely continuous with an independent loop in an
independent soup. Using the decomposition of the three-dimensional loop measure
into bubble measures \cite[Proposition~4.2]{JL26}, Proposition~\ref{dimension of a single Brownian loop}
therefore transfers to the full loop soup. Taking the countable union over the
truncation parameters gives the almost sure upper bound, while the same
comparison gives the positive-probability lower bound.
\end{proof}

We next state the ergodicity and zero-one law for the three-dimensional
BLS. A discrete version can be found in \cite[Proposition~3.2]{CS16}; the
continuum version below follows by a standard argument, as in
\cite[Lemma~6.1]{JL26}. Let $\mathcal{G}$ be the sigma-field generated by
$\Lc^\alpha$. Let $\tau_x:\mathcal{G}\to\mathcal{G}$ be the shift operator
induced by the translation $y\mapsto x+y$ in $\mathbb{R}^3$.
\begin{lemma}[Ergodicity]\label{Ergodicity}
    For all $x\in \Rb^3 \setminus \{0\}$, the shift operator $\tau_x$ is
    ergodic with respect to the law of $\Lc^\alpha$. That is, the law of
    $\Lc^\alpha$ is preserved under $\tau_x$, and for every
    $\tau_x$-invariant event $A\in \mathcal{G}$, we have
    $\Pb(\Lc^\alpha \in A)\in\{0,1\}$.
\end{lemma}

We now prove Theorem~\ref{dimension}.
\begin{proof}[Proof of Theorem \ref{dimension}]
    By Proposition~\ref{Estimate of dimension},
    \[
        \dim_{\mathcal{H}}(G_{\rm loc})
        \le
        \max \{2-\xi_{\alpha}(1,1),0 \}
    \]
    almost surely. It remains to prove the corresponding lower bound.
    The event
    \[
        \left\{
        \dim_{\mathcal{H}}(G_{\rm loc})
        \ge
        \max \{2-\xi_{\alpha}(1,1),0 \}
        \right\}
    \]
    is translation invariant and occurs with positive probability. By
    Lemma~\ref{Ergodicity}, it in fact occurs with probability one. This proves the
    theorem.
\end{proof}
\section{Continuity of generalized intersection exponent}

In this section, we prove the continuity of the generalized intersection
exponent at intensity zero. The only input about small-intensity BLS clusters
is the following consequence of the comparison with Mandelbrot fractal
percolation.

\begin{proposition}\label{prop: all clusters are small}
There exists \(R>1\) such that
\[
    \lim_{\alpha\downarrow 0}
    \mathbb P\bigl(
    \partial \mathcal{B}_1\overset{\mathcal L^\alpha}{\longleftrightarrow} \partial \mathcal{B}_R
    \bigr)
    =0 .
\]
\end{proposition}

\begin{proof}
We use the comparison between three-dimensional Brownian loop-soup clusters
and Mandelbrot fractal percolation from \cite[Lemma~6.8]{JL26}. In that
comparison, a loop-soup cluster crossing a fixed annulus is ruled out by the
existence of a separating sheet in an associated Mandelbrot percolation model,
whose retention parameter tends to \(1\) as \(\alpha\downarrow0\).

The fractal-percolation input used in \cite{JL26} is stated there in a
positive-probability form, but the original Mandelbrot percolation argument
gives the stronger limiting statement that the probability of such a
separating sheet tends to \(1\) as the retention parameter tends to \(1\);
see \cite{CCD88,CCGS91}. Combining this limiting form with the comparison of
\cite[Lemma~6.8]{JL26} gives the desired conclusion.
\end{proof}

We are now ready to prove Theorem~\ref{theorem: continuity of intersection exponent}.
\begin{proof}[Proof of Theorem~\ref{theorem: continuity of intersection exponent}]
    For simplicity, we prove the theorem only in the case $k=1$; the
    argument for general $k$ is the same. For $\beta\ge0$, write
    $\Lambda_{r,\beta}$ and $Z_{r,\beta}$ to display the dependence on the
    intensity. We have
    \[
    Z_{r,0}-Z_{r,\alpha}
    =
    \Pb\bigl(
    \overline{B}_r \cap \Lambda_{r,0} = \emptyset,\,
    \overline{B}_r \cap (\Lambda_{r,\alpha}\setminus \Lambda_{r,0}) \ne \emptyset
    \mid \mathcal{F}_r
    \bigr).
    \]
    Let $E_{\delta,\alpha}$ be the event that every cluster of $\Lc^\alpha$
intersecting $\mathcal{B}_{e^r}$ has diameter less than $\delta$.
    Using the same argument as in Lemma~\ref{lem:cluster_small 1}, we can cover
    $\mathcal{B}_{e^r}$ by finitely many annuli congruent to    $\mathcal{B}_{\delta}\setminus \mathcal{B}_{\delta/R}$ such that if
    $E_{\delta,\alpha}$ does not occur, then there exists a cluster crossing one of
    these annuli, where $R$ is given in Proposition~\ref{prop: all clusters are small}. By Proposition~\ref{prop: all clusters are small}, we have
    \[
    \lim_{\alpha\downarrow 0}\Pb(E_{\delta,\alpha})=1.
    \]
    Separating according to whether $E_{\delta,\alpha}$ occurs, we get
    \[
    Z_{r,0}-Z_{r,\alpha}
    \le
    \Pb\bigl(
    0<\mathrm{dist}(\overline{B}_r,\Lambda_{r,0})\le\delta,\,
    E_{\delta,\alpha}
    \mid \mathcal{F}_r
    \bigr)
    +\mathbf{1}_{E_{\delta,\alpha}^c}.
    \]
    Letting $\delta\to0$, the first term tends to $0$ almost surely. Since
    $0\le Z_{r,\alpha}\le 1$, we obtain by dominated convergence that
    \[
    \lim_{\alpha\downarrow0}\mathbb{E}[Z_{r,\alpha}^{\lambda}]
    =
    \mathbb{E}[Z_{r,0}^{\lambda}].
    \]

    Since $\xi_\alpha(1,\lambda)$ is monotone in $\alpha$, the limit
    $\lim_{\alpha\downarrow0}\xi_\alpha(1,\lambda)$ exists.
    Fix $\alpha_*\in I_0$ and $\lambda_0<\lambda<\lambda_1$. By
    Theorem~\ref{thm:up-to-constant}, for all $\alpha\in(0,\alpha_*]$,
    
    \[
    C_1 e^{-r\xi_\alpha(1,\lambda)}
    \le
    \mathbb{E}[Z_{r,\alpha}^{\lambda}]
    \le
    C_2 e^{-r\xi_\alpha(1,\lambda)},
    \]
    where $C_1,C_2>0$ depend only on $\alpha_*,\lambda_0,\lambda_1$.
    Letting $\alpha\downarrow0$, we get
    \[
    C_1 e^{-r\lim_{\alpha\downarrow0}\xi_\alpha(1,\lambda)}
    \le
    \mathbb{E}[Z_{r,0}^{\lambda}]
    \le
    C_2 e^{-r\lim_{\alpha\downarrow0}\xi_\alpha(1,\lambda)}.
    \]
    Letting $r\to \infty$ and combining this with the classical up-to-constants estimate for Brownian
    intersection exponents from \cite{Lawler96}, we conclude that
    $\lim_{\alpha\downarrow0}\xi_\alpha(1,\lambda)=\xi(1,\lambda)$.
    
   Finally, by the bound $
    1/2\le \xi(1,1)<1$ from
    \cite{BL90}, and continuity at $\alpha=0$, there exists $\alpha_1>0$ such that
    $\xi_\alpha(1,1)<1$ for all $\alpha\in(0,\alpha_1)$. Therefore,
    \[
    {\rm dim}_{\mathcal{H}}(G_{\rm loc})=2-\xi_\alpha(1,1)>1
    \]
    for all $\alpha\in(0,\alpha_1)$.
\end{proof}
We end with the proof of Corollary~\ref{coro:percolation dimension}.
\begin{proof}[Proof of Corollary~\ref{coro:percolation dimension}]
For \(s\in(0,1)\), call \(\widetilde B(s)\) a generalized cut point if
    \[
        \Lambda\bigl(\widetilde B[0,s)\bigr)\cap \widetilde B(s,1]=\emptyset .
    \]
    Let \(L\) be the set of generalized cut points. The first-moment and
    second-moment argument of \cite{Lawler96} applies verbatim, with
    Theorem~\ref{thm:up-to-constant} replacing the classical up-to-constants
    estimate for Brownian intersection exponents. Thus
    \[
        \dim_{\mathcal H}(L)
        =
        \max \{2-\xi_{\alpha}(1,1),0 \}
        \qquad \text{a.s.}
    \]
    We next observe that every admissible curve must contain these points.
    Indeed, fix \(s\) with \(\widetilde B(s)\in L\). The set
    \(\Lambda(\widetilde B[0,s))\) contains the starting point \(\widetilde B(0)\) and,
    by the defining property of \(L\), is disjoint from the future path
    \(\widetilde B(s,1]\). Hence any continuous curve in
    \(\Lambda(\widetilde B[0,1])\) from \(\widetilde B(0)\) to \(\widetilde B(1)\) must pass
    through \(\widetilde B(s)\); otherwise it would give a connection from
    \(\Lambda(\widetilde B[0,s))\) to \(\widetilde B(s,1]\), contradicting the
    definition of \(L\). Therefore \(L\subseteq\gamma[0,1]\) for every
    \(\gamma\in\Gamma\).

    Consequently, for every \(\gamma\in\Gamma\),
    \[
        \dim_{\mathcal H}(\gamma)
        \ge
        \dim_{\mathcal H}(L)
        =
        \max \{2-\xi_{\alpha}(1,1),0 \}.
    \]
    By Theorem~\ref{theorem: continuity of intersection exponent} and the
    strict bound \(\xi(1,1)<1\) from \cite{BL90}, after decreasing
    \(\alpha_1\) if necessary, we have
    \(\xi_\alpha(1,1)<1\) for all \(\alpha\in(0,\alpha_1)\). Hence
    \(\dim_{\mathcal H}(\gamma)>1\) for every \(\gamma\in\Gamma\), and taking
    the infimum over \(\Gamma\) gives \(\zeta>1\) almost surely.
\end{proof}

\bibliographystyle{abbrv}
\bibliography{references}

\end{document}